\documentclass[12pt]{article}
\usepackage{amsmath, amsthm, amssymb}
\usepackage{graphicx,psfrag,epsf}
\usepackage{enumerate}
\usepackage{natbib}
\usepackage{multirow}
\usepackage{url} 

\newcommand{\blind}{1}
\newcommand{\cov}{\ensuremath{ {\rm cov}}}
\newcommand{\cor}{\ensuremath{ {\rm cor}}}
\newcommand{\var}{\ensuremath{ {\rm var}}}
\newcommand{\E}{\mathbb{E}}
\theoremstyle{plain} \newtheorem{prop}{{\sc Proposition}}

\begin{document}

\def\spacingset#1{\renewcommand{\baselinestretch}%
{#1}\small\normalsize} \spacingset{1}


\if1\blind
{
  \title{\bf Reallocating and Resampling: A Comparison for Inference}
  \author{Kari Lock Morgan 
  \hspace{.2cm}\\
    Department of Statistics, Pennsylvania State University\\
    }
  \maketitle
} \fi

\if0\blind
{
  \bigskip
  \bigskip
  \bigskip
  \begin{center}
    {\LARGE\bf Reallocating and Resampling: A Comparison for Inference}
\end{center}
  \medskip
} \fi

\bigskip
\begin{abstract}
Simulation-based inference plays a major role in modern statistics, and often employs either reallocating (as in a randomization test) or resampling (as in bootstrapping).  Reallocating mimics random allocation to treatment groups, while resampling mimics random sampling from a larger population; does it matter whether the simulation method matches the data collection method?  Moreover, do the results differ for testing versus estimation?  Here we answer these questions in a simple setting by exploring the distribution of a sample difference in means under a basic two group design and four different scenarios: true random allocation, true random sampling, reallocating, and resampling.  For testing a sharp null hypothesis, reallocating is superior in small samples, but reallocating and resampling are asymptotically equivalent.  For estimation, resampling is generally superior, unless the effect is truly additive.  Moreover, these results hold regardless of whether the data were collected by random sampling or random allocation.  
\end{abstract}

\noindent%
\vfill

\newpage
\spacingset{1.45} 

\section{Introduction}
\label{intro}

Simulation-based inference commonly takes two different forms: reallocating or resampling.  Reallocating mimics random allocation to treatment groups, while resampling mimics random sampling from a greater population.  Moreover, reallocating implicitly assumes a null hypothesis, while resampling does not.  Should the distinction for when to use one simulation method over the other be based on the original data collection method, the inferential goal (testing or estimation), or neither?  When, and to what extent, does it matter?  Through mathematical derivation, we find that resampling is better for interval estimation unless the treatment effect is truly additive, reallocating is better for testing in small samples, and the two are asymptotically equivalent for testing in large samples.  Also, quite surprisingly, and contrary to conventional wisdom, these results hold regardless of whether the data were obtained by random allocation or random sampling.

Although the phrase ``resampling methods" is often used for any simulation-based method, here we use ``resample" to mean take a random sample with replacement from the original sample, of the same sample size as the original sample, as in a nonparametric bootstrap \citep{efron79}.  We use ``reallocate" to refer to the reallocation of outcomes to treatment groups, as is done in the classic permutation or randomization test \citep{fisher35, edgington07}.  Reallocating is also commonly referred to as rerandomizing, permuting, shuffling, or scrambling.   Reallocating conditions on the fixed sample and samples without replacement, whereas resampling assumes a greater population and samples with replacement from the observed sample. 

Reallocating originated from the design and analysis of experiments as a method for testing, dating back to Fisher's historic text {\em The Design of Experiments} \citep{fisher35}.  It was originally proposed as a way to measure significance solely based on the random allocation of an experiment. However, only a year later, Fisher discusses reallocating in the context of random sampling, while the emphasis on testing remained \citep{fisher36}.   Thus it appears that Fisher viewed reallocating as applicable regardless of the type of randomness in data collection, but primarily as a method for testing.   Others \citep[Section 1.4]{tukey93, manly07} have since adapted reallocation for estimation by inverting a series of tests and creating an interval comprised of values not rejected.  

Resampling originated with Efron's landmark paper on bootstrapping \citep{efron79}, which addresses only data obtained via random sampling, and focuses almost entirely on estimation.  Efron and others have since gone on to apply resampling for testing \citep[Section 16.2]{efron94}, but the assumption of random sampling remains.

Advice regarding when either simulation method is appropriate differs.  While Fisher clearly regarded reallocating as applicable under both random allocation and random sampling, \citet{tukey88} states that reallocating is applicable only under random allocation.  \citet{tukey88} also says that resampling is applicable under either random sampling or allocation, but others advocate for reallocating as a way to bypass the assumption of random sampling that is often unrealistic for randomized experiments \citep{cotton73, bredenkamp80, ludbrook98}, suggesting that other methods may not be applicable without random sampling.  \citet{lafleur09} recommend reallocating for testing and resampling for estimation, and \citet{ludbrook95, ludbrook98} recommend reallocating for testing under random allocation, resampling for testing under random sampling, and resamping for estimation under either data collection method.  Other comparisons prove that resampling and reallocating are asymptotically equivalent for testing \citep{romano89}, reallocating is more powerful than resampling for testing \citep{donegani91, lafleur09}, and resampling performs poorly in small sample sizes \citep{donegani91, hall91, efron94}.  A good historical summary comparing the two methods is given in \citet[Section 1.3.5]{berry14}.

While reallocating originated for testing after random allocation and resampling originated for estimation after random sampling, both methods are widely used often irrespective of the data collection method or inferential goal.  Moreover, while authors are in general agreement about the use of reallocating for tests based on experiments and the use of resampling for estimation on random samples, there is discrepancy as to when (or whether) it matters, and whether the data collection method or the inferential goal should take priority if the two don't align.  To the best of our knowledge, this appears to be a void in the literature that deserves to be addressed.  Moreover, most of the existing comparisions are based on conceptual, rather than mathematical, foundations (\citet{romano89} and \citet{donegani91} are notable exceptions) and rigorous mathematical justifications are needed for determining not just whether reallocating or resampling is preferred in a given situation, but also the extent to which the choice matters.  

This paper addresses these questions by comparing the distribution of the common difference in means under different scenarios; true random sampling, true random allocation, resampling, and reallocating.  These are all examined under different contexts: testing a sharp null hypothesis of no difference (Section \ref{sharpnull}), estimation (Section \ref{intervals}), and testing a weaker null hypothesis of equal means (Section \ref{equalmeans}).  Section \ref{examples} illustrates with two examples, one with data generated via random allocation and one with data generated via random sampling. Section \ref{summary} summarizes the findings and concludes.

\section{General Framework}\label{framework}

This paper examines a quantitative outcome, $Y$, measured on $n$ units with binary group membership, ${\bf W} = (W_1, \dots, W_n)$ where {$W_i \in \{0, 1\}$.  To focus primarily on the distinction between random allocation and random sampling, we restrict our comparison to two simple, but widely used, settings: a completely randomized experiment and a simple stratified sample, both with two groups of fixed sample sizes; $n_1 \equiv \sum_{i=1}^n W_i$ and $n_0 \equiv \sum_{i=1}^n (1-W_i)$.  It is important to note that any results here are applicable only to these two simple designs. 

Many note that random allocation and random sampling lead to two fundamentally different modes of inference, given different names by different authors: experimental versus sampling inference \citep{kempthorne79}, randomization versus population inference \citep{ludbrook95}, finite sample versus super population inference \citep{imbens15}, and permutation versus population model \citep{berry14}.  The former, stemming from random allocation and addressing only the sample at hand, originated with \citet{fisher35, fisher36}.  The latter, stemming from random sampling and generalizing to a larger population, originated with \citet{neyman28}.  \citet{kempthorne79} argues that it is misleading to refer to both using the same single word ``inference".  Here we are explicit about this distinction in our notation.   

When random allocation is the source of randomness, we condition on the units in the sample, and are interested in what might have happened, had they been assigned a different treatment.   Here we follow the potential outcome notation of the Rubin Causal Model \citep{rubin74}, and let $Y_i(W_i)$ denote the $i^{th}$ unit's potential outcome under treatment assignment $W_i$.  For this notation to make sense, we assume the standard Stable Unit Treatment Value Assumption (SUTVA) \citep{rubin80}; that ${\bf W}$ has only two well-defined levels and that potential outcomes for unit $i$ depend only on $W_i$, and not on the treatment assignments of other units. We assume here that the potential outcomes are fixed, and only ${\bf W}$ is random.  Let ${\bf Y}(w) = (Y_1(w), \dots, Y_n(w))$ denote the vectors of potential outcomes under treatment $w$, and ${\bf Y}^{obs} = (Y_1^{obs}, \dots, Y_n^{obs})$ denote the vector of observed outcomes, where $Y_i^{obs} = Y_i(1)W_i  +  Y_i(0) (1-W_i).$

When random sampling is the source of randomness, we are now primarily interested in the underlying distributions for the population from which our sample was drawn.  We thus alter our notation slightly, although for comparison purposes attempt to keep it as similar as possible.  In this context, let $Y(w)$ be a random variable denoting outcomes from group $w$, and using the square bracket notation of [mean, variance], we have $Y_i(w) \sim \left[ \mu_w, \sigma_w^2 \right]$, for $w \in \{0, 1\}$.  As the analog to SUTVA, here we assume independence between units.  

Although the notational distinction here is framed in terms of random allocation versus random sampling, the distinction may be better viewed in terms of causality or generalizability.   In the population framework, we are asking a question inherently about generalizability:  (how much) do the population distributions differ between the groups?   In the potential outcomes framework, we are asking an inherently causal question: (how much) would outcomes change under the opposite treatment?  The potential outcomes framework is useful for causal inference in non-experimental settings as well \citep{rubin07}, although without random allocation confounding variables must be accounted for, for example with propensity score methods \citep{rosenbaum83}.  \citet{rosenbaum84} incorporated the propensity score with reallocating for testing causality in observational studies, but here for simplicity we restrict our causal conclusions to the experimental framework.

While these methods allow for any statistic, throughout we use the difference in means:
\begin{equation}\label{tauhat}
\hat{\tau} \equiv \overline{Y}^{obs}(1) - \overline{Y}^{obs}(0)  \equiv \frac{\sum_{i=1}^n Y_i^{obs} W_i}{n_1} - \frac{\sum_{i=1}^n Y_i^{obs} (1-W_i)}{n_0}.
\end{equation}

Because there are often several different methods for getting from the distribution of $\hat{\tau}$ to a p-value or confidence interval, we restrict our focus here to the distribution of $\hat{\tau}$.  For all methods, the distribution will be centered at $0$ for testing and $\hat{\tau}$ for estimation, assuming symmetry.  Also, although higher moments of $\hat{\tau}$ can certainly be important as discussed in Section~\ref{acs}, under all methods considered here $\hat{\tau}$ will be approximately normally distributed if $n_0$ and $n_1$ are large enough \citep{li17}.  For these reasons, we focus henceforth on $\var(\hat{\tau})$.  

As a crucial point, although the two modes of inference differ in notation and scope of inference, in both cases we work with the same statistic: $\hat{\tau}$.  Although random allocation or random sampling may yield different statistics (in part due to the issue of confounding with the latter), for a given data set $\hat{\tau}$ is calculated agnostic to the mode of inference.  Therefore, any differences in $\var(\hat{\tau})$ arise solely due to differences in data collection (random allocation versus random sampling), differences in simulation (reallocating versus resampling), or differences due to testing versus estimation.

\section{Testing a Sharp Null of No Difference}\label{sharpnull}

In this section we consider tests of Fisher's sharp null hypothesis \citep{fisher35} of no difference whatsoever.  Under the experimental framework this can be interpreted as no treatment effect for any unit, $H_0: {\bf Y}(1) = {\bf Y}(0)$, and under the sampling framework this can be interpreted as no difference in the distributions for the two groups, $H_0: Y(1) \sim Y(0)$, or equivalently, all outcomes come from the same distribution.   In either case, the sharp null implies that the outcomes are independent of ${\bf W}$.  A weaker null hypothesis of no mean difference will be considered in Section~\ref{equalmeans}.

In the subsections that follow, we derive $\var(\hat{\tau})$ under random allocation, random sampling, reallocating, and resampling; under the sharp null hypothesis in this section, and for estimation in Section~\ref{intervals}.  The variance of $\hat{\tau}$ is highlighted within each subsection as a proposition to help the reader distill the key results.  The goal of this paper is to synthesize otherwise disparate results, providing $\var(\hat{\tau})$ for each of the different scenarios in a cohesive and consistent format, and under as comparable as possible conditions, allowing us to determine the extent to which the choice of simulation method matters, the extent to which the simulation method should match the corresponding data collection method, and the extent to which the chosen simulation method should differ for testing and estimation.

\subsection{Random Allocation and a Sharp Null Hypothesis}\label{ratest}

 \begin{prop}
 Under random allocation to two groups of fixed sample sizes $n_0$ and $n_1$, and assuming the sharp null hypothesis $H_0: {\bf Y}(1) = {\bf Y}(0)$, 
 \begin{equation}
 \var \left( \hat{\tau} \right) = s^2 \left( \frac{1}{n_1} + \frac{1}{n_0} \right), \label{ranullvar}
 \end{equation}
 where 
$s^2 \equiv \frac{1}{n-1} \sum_{i=1}^n (Y_i^{obs} - \overline{Y}^{obs})^2$ is the sample variance of ${\bf Y}^{obs}$.
 \end{prop}

\begin{proof}
Define $p \equiv P(W_i = 1) = n_1/n$, and we derive the covariance of $\bf{W}$ as follows: 
\begin{align*}
\text{ If }i=j :   \cov(W_i,W_j) & =\var(W_i) = p(1-p) \\ 
\text{ If }i \neq j  : \cov (W_i,W_j)&=\E(W_i W_j)-\E(W_i)\E(W_j)  =\frac{-p(1-p)}{n-1},  
\end{align*}
so
\begin{equation} \label{covW}
\cov({\bf W}) = \frac{p(1-p)}{n-1} \left( n{\bf I}_{n \times n} - {\bf 1}_{n\times n}\right),
\end{equation}
where ${\bf 1}_{n\times n}$ is the $n \times n$ matrix of all 1s, and ${\bf I}_{n \times n}$ is the $n \times n$ identity matrix. 

Define 
\begin{equation}\label{swstar}
{s^*_w}^2 \equiv \frac{ \sum_{i=1}^n \left(Y_i(w) - \overline{Y}(w) \right)^2}{n - 1},
\end{equation}
to be the sample variance for ${\bf Y}(w)$ for all units in the sample (regardless of actual treatment assignment).  Then, recalling ${\bf Y}(w)$ is fixed and only ${\bf W}$ is random, for $w \in \{0, 1\}$,
\begin{align}
\var \left( {\bf Y}(w)'{\bf W} \right) &= {\bf Y}(w)' \cov({\bf W}) {\bf Y}(w) \notag \\
&=  \frac{np(1-p)}{n-1} \left( \sum_{i=1}^n (Y_i(w) - \overline{Y}(w))^2 \right) \notag \\
&=  {s_w^*}^2 np(1-p).
\end{align}

Likewise, $\var \left( {\bf Y}(w)' ({\bf 1}_n - {\bf W}) \right) = {s_w^*}^2 np(1-p)$.  So for $w \in \{0, 1\}$,
\begin{equation}
\var \left(\overline{Y}^{obs}(w) \right) = \frac{{s_w^*}^2 np(1-p)}{n_w^2} = \frac{{s_w^*}^2}{n_w} \left( \frac{n - n_w}{n} \right).
\end{equation}

Note that $(n - n_w)/n$ is almost the familiar finite population correction factor, except with an $n$ rather than $n-1$ in the denominator.  This is because we defined ${s^*_w}^2$ with an $n-1$ in the denominator, despite the calculation being over the entire finite ``population".  

Under the sharp null hypothesis, ${\bf Y}(0) = {\bf Y}(1) = {\bf Y}^{obs}$, so ${s^*_1}^2 = {s^*_0}^2 = s^2$  and 
 \begin{equation}\label{rameanvar}
 \var \left(\overline{Y}^{obs}(w) \mid {\bf Y}(1)  = {\bf Y}(0) \right) = \frac{s^2}{n_w} \left( \frac{n - n_w}{n} \right).
 \end{equation}

Also under the sharp null, $\sum_{i=1}^n Y_i^{obs} W_i + \sum_{i=1}^n Y_i^{obs} (1-W_i) = \sum_{i=1}^n Y_i^{obs}$, so
\begin{equation}
\cor \left( \overline{Y}_1, \overline{Y}_0  \mid {\bf Y}(1) = {\bf Y}(0)\right) = -1.
\end{equation}
Therefore, under random allocation 
\begin{align*}
\var \left( \hat{\tau}   \mid {\bf Y}(1) = {\bf Y}(0)\right) &= \var \left( \overline{Y}_1 \right) + \var \left( \overline{Y}_0 \right) - 2 \cov \left( \overline{Y}_1, \overline{Y}_0 \right) \notag \\
&= \frac{s^2(1-p)}{np} + \frac{s^2 p}{n(1-p)} + 2 \frac{s^2}{n} \left( \frac{(1-p)p}{p(1-p)} \right)^{1/2} \notag \\
&= s^2 \left( \frac{1}{n_1} + \frac{1}{n_0} \right). 
\end{align*}
\end{proof}


Unlike most true variance calculations that rely on unknown quantities, this can actually be calculated (not just estimated) from the observed data. This differs from the typical formula found in most introductory statistics textbooks because the sharp null hypothesis assumes a common variance between treatment groups.  This also differs from the usual pooled variance estimate, as $s^2$ is calculated assuming both equal variances and equal means between the groups, so considers deviations of each data point from the grand mean, as opposed to deviations of each data point from its own group mean.

 
 \subsection{Reallocating for Testing a Sharp Null Hypothesis}\label{reallocatetest}
 
Under the sharp null hypothesis, ${\bf Y}^{obs} = {\bf Y}(1) = {\bf Y}(0)$ is the same for every treatment assignment ${\bf W}$.  Thus we can empirically create the distribution of $\hat{\tau}$ under the null by leaving ${\bf Y}^{obs}$ untouched and simulating many random allocations, ${\bf W}$, and calculating $\hat{\tau}$ for each simulated ${\bf W}$.  Once a null distribution is generated in this fashion, the p-value is calculated as the proportion of simulated statistics as extreme as, or more extreme than, the observed difference in means \citep{fisher35, pitman37}.   Note that in a randomization test the variance of $\hat{\tau}$ is not used or needed, as the p-value is simply computed empirically as a proportion of simulated statistics, but we give it here for comparison purposes.  

If all random allocations were to be enumerated (exact method) the null distribution achieved through reallocation in this fashion is exactly the true null distribution under random allocation, and thus the p-value achieved in this fashion is widely recognized as the gold standard correct p-value \citep{fisher36, kempthorne52, bradley68, tukey88, edgington07}.  However, in most realistic cases, enumerating and calculating the statistic for all ${n \choose n_1}$ allocations is not feasible, so instead we simulate a large number of reallocations (Monte Carlo method).  In this case, the randomization distribution will converge to the true null distribution as the number of simulations increases.   

 \begin{prop}\label{randallocatetest}
 Under reallocation to two groups of fixed sizes $n_0$ and $n_1$, and assuming the sharp null hypothesis $H_0: {\bf Y}(1) = {\bf Y}(0)$, as the number of simulations increases,
 \begin{equation}
 \var \left( \hat{\tau} \right) \rightarrow s^2 \left( \frac{1}{n_1} + \frac{1}{n_0} \right). \label{reallocatenullvar}
 \end{equation}
 \end{prop}

\begin{proof}
As discussed, under the sharp null hypothesis, the distribution of the statistic under reallocation will equal the true null distribution under random allocation, either exactly (exact method) or in the limit (monte carlo method).  Therefore, $\var(\hat{\tau})$ is as in \eqref{ranullvar}.
\end{proof}


 
 \subsection{Random Sampling and a Sharp Null Hypothesis}\label{rstest}

\begin{prop}
 Under random sampling of two groups of fixed sample sizes $n_0$ and $n_1$, and assuming the sharp null hypothesis $H_0: Y(0) \sim Y(1)$, 
 \begin{equation}
 \var \left( \hat{\tau} \right) = \sigma^2 \left( \frac{1}{n_1} + \frac{1}{n_0} \right), \label{rsnullvar}
 \end{equation}
 where $\sigma^2$ is the population variance.
 \end{prop}
 
 \begin{proof}
Under fixed sample sizes, simple random sampling, and independence, it is straightforward to derive the familiar formula: 
\begin{equation}\label{rsvar}
\var \left( \hat{\tau} \right) = \frac{\sigma_1^2}{n_1} + \frac{\sigma_0^2}{n_0},
\end{equation}
where $\sigma_w^2$ denotes the population variance within group $w$.  

Under the sharp null hypothesis, $\sigma_1 = \sigma_0 = \sigma$, yielding
\begin{equation*}
\var \left(\hat{\tau}  \mid Y(1) \sim Y(0)\right) = \sigma^2 \left( \frac{1}{n_1} + \frac{1}{n_0}\right).
\end{equation*}
\end{proof}

Therefore, under the sharp null hypothesis, the distributions of $\hat{\tau}$ under random allocation or random sampling are almost equivalent, only differing by $s^2$ versus $\sigma^2$ in $\var(\hat{\tau})$.


\subsection{Resampling for Testing a Sharp Null Hypothesis}\label{resampletest}

To resample under the sharp null hypothesis, select samples of size $n_1$ and $n_0$, each sampled with replacement from the combined sample of all $n$ units.  Repeating this process many times and calculating $\hat{\tau}$ for each yields a distribution of $\hat{\tau}$ under the sharp null, and as with reallocating, a p-value can then be calculated as the proportion of the simulated statistics that are as extreme as (or more extreme than) than the observed statistic.  There are other ways of obtaining a p-value from this distribution \citep{hall91}, but for comparison we focus here on $\var(\hat{\tau})$ under the sharp null.

 \begin{prop}
 Under resampling of two groups of fixed sample sizes $n_0$ and $n_1$, and assuming the sharp null hypothesis $H_0: Y(0) \sim Y(1) $, as the number of simulations increases,
 \begin{equation}
 \var \left( \hat{\tau} \right) \rightarrow s^2  \left( \frac{1}{n_1} + \frac{1}{n_0} \right) \left( \frac{n-1}{n} \right). \label{resnullvar}
 \end{equation}
\end{prop}

\begin{proof}
As derived by \citet[pg. 43]{efron94},  
\begin{equation}
\var \left( \overline{Y}^{obs}(w) \mid Y(1) \sim Y(0) \right) \rightarrow \frac{s^2}{n_w} \left( \frac{n-1}{n} \right). \label{bootmeanvar}
\end{equation}
Because the resampled groups will be independent (due to sampling with replacement),
\begin{equation}
\var \left( \hat{\tau} \mid Y(1) \sim Y(0) \right) \rightarrow s^2  \left( \frac{1}{n_1} + \frac{1}{n_0} \right) \left( \frac{n-1}{n} \right). \notag
\end{equation}
\end{proof}

  Thus resampling gives a variance estimate that is biased in small samples.  The variance will on average be slightly lower than the variance under random allocation, random sampling, or reallocating.  However, the bias diminishes as $n$ gets large, and the estimate is consistent under either data collection method.  Moreover, the variances under resampling, \eqref{resnullvar}, or reallocating, \eqref{reallocatenullvar}, differ only by $(n-1)/n$, and hence are asymptotically equivalent.
  
  Note that reallocating and resampling are actually quite similar, with only one notable difference.  In either case we take a sample from ${\bf Y}^{obs}$, set $W_i = 1$ for $i = 1, \dots, n_1$ and $W_i = 0$ for $i = n_1 + 1, \dots, n$, calculate $\hat{\tau}$, and repeat many times. The only difference is whether this sample is taken with (resampling) or without (reallocating) replacement.  This perspective can also help explain the extra factor of $(n-n_1)/n$ in \eqref{rameanvar} as the adjustment for sampling without replacement.  \citet{donegani91} shows that asymptotically, sampling without replacement is better than sampling with replacement from a finite population, in the sense that means will generally be closer to the overall mean.  
 
 \section{Interval Estimation}\label{intervals}
This section mimics Section~\ref{sharpnull},  but in the context of estimation, without a null hypothesis.      

We first define the estimands, which differ depending on the mode of inference.  Assuming we are interested in the true average difference between groups, the distinction becomes whether this average difference is over all units in the sample, if both potential outcomes could be observed, or over all units in the population, if we could collect data on the entire population.  When the only source of randomness is random allocation, we take our estimand to be the true average treatment effect for the sample:
\begin{equation}\label{taufs}
\tau_{ra} = \frac{1}{n} \sum_{i=1}^n \left[Y_i(1) - Y_i(0) \right].
\end{equation}
Under random sampling from a greater population, our estimand can be defined as 
\begin{equation}\label{spestimand}
\tau_{rs} = \E \left[ Y(1) - Y(0) \right] = \mu_1 - \mu_2.
\end{equation}

Although it does not change the resulting confidence interval, which are based only on the statistic $\hat{\tau}$, note that the distinction between $\tau_{ra}$ and $\tau_{rs}$ is not just finite sample versus population inference, but again whether the inferences are inherently causal in nature.  

\subsection{Random Allocation without a Null Hypothesis}

\begin{prop}
 Under random allocation to two groups of fixed sample sizes $n_0$ and $n_1$, 
 \begin{equation}
\var \left( \hat{\tau}  \right) = \frac{{s^*_1}^2}{n_1} + \frac{{s^*_0}^2}{n_0} - \frac{{s^*_{1-0}}^2}{n}, \label{ravar}
 \end{equation}
 where ${s^*_w}^2$ is defined in \eqref{swstar} and \begin{equation}\label{s10}
{s^*_{1-0}}^2 \equiv \frac{\sum_{i=1}^n \left( Y_i (1) - Y_i(0) - (\overline{Y}(1) - \overline{Y}(0)) \right)^2}{n-1}
\end{equation}
is the sample variance of the unit-level treatment effects.  
 \end{prop}
 
This result is derived in \citet[Theorem 6.2]{imbens15}.

We make three important observations regarding \eqref{s10} before proceeding.  The first is that ${s^*_{1-0}}^2$ is always positive, so ignoring this third term in \eqref{ravar}, as is typically done, will provide a conservative (overestimate) of the true randomization variance.  The second is that ${s^*_{1-0}}^2$ can never be directly calculated, and usually cannot even be estimated, because we never observe both potential outcomes for the same unit.  Neyman \citep[translated from the 1923 Polish version]{neyman90} recognized both facts, and so chose to define confidence intervals as having {\em at least} the specified level of coverage.  Lastly,  if the treatment effect is additive, ${\bf Y}(1) = {\bf Y}(0) + a$, then ${s^*_{1-0}}^2 = 0$ and 
 \begin{equation} \label{additive}
\var \left(\hat{\tau}  \mid {\bf Y}(1) = {\bf Y}(0) + a \right) = \frac{{s^*_1}^2}{n_1} + \frac{{s^*_0}^2}{n_0} 
= {s^*_1}^2 \left(\frac{1}{n_1} + \frac{1}{n_0} \right) 
= {s^*_0}^2 \left(\frac{1}{n_1} + \frac{1}{n_0} \right).
\end{equation}

Note that in general even if  ${s^*_1}^2 = {s^*_0}^2$, they will not equal $s^2$ (except in special cases), because $s^2$ is calculated via taking deviations from the grand mean, whereas each ${s^*_w}^2$ is calculated taking deviations from its group mean.   The sharp null hypothesis is a special case of additivity, in which ${s^*_1}^2 = {s^*_0}^2 = s^2$ because the group means are assumed to be the same, and in this case \eqref{additive} yields an alternative derivation of \eqref{ranullvar}.

 \subsection{Reallocation for Interval Estimation}\label{reallocateint}
 
 Reallocating requires an assumed relationship between $Y_i(1)$ and $Y_i(0)$ such that regardless of which treatment is observed for each unit, both potential outcomes can be assumed known.  The most commonly assumed relationship is that of an additive treatment effect: ${\bf Y}(1) = {\bf Y}(0) + a$.  \citet{tukey93} acknowledges that with reallocating ``we will be pairing platinum standard significance tests with silver or pewter standard confidence intervals" due to the necessity of this additivity assumption, stating that ``Such rigid displacement is unlikely to happen in practice. It is a less than completely verifiable assumption."  Moreover, unlike most of the results from this paper, this idea does not naturally extend to discrete data, as it would require adding fractional values, potentially resulting in nonsensical values.  

However, under additivity, reallocating can be used to generate a confidence interval as the values of $a$ that would not be rejected under a test of the null hypothesis $H_0: {\bf Y}(1) = {\bf Y}(0) +a$ \citep[Section 1.4]{pitman37, manly07}. This can be quite computationally intensive; \citet{garthwaite96} provides an algorithm for doing this more efficiently.  

 \begin{prop}
 Under reallocation to two groups of fixed sample sizes $n_0$ and $n_1$, and assuming $H_0: {\bf Y}(1) = {\bf Y}(0) + a$, as the number of simulations increases, 
  \begin{equation}
 \var \left( \hat{\tau} \right) \rightarrow \left[ s^2  + \frac{n_0 n_1 a}{n (n-1)} \left(a - 2\hat{\tau} \right) \right]  \left( \frac{1}{n_1} + \frac{1}{n_0} \right). \label{savar}
 \end{equation}
\end{prop}

\begin{proof}
  Define $s_a^2 \equiv {s_1^*}^2 ={s_0^*}^2$ to be the common within group variance under $H_0: {\bf Y}(1) = {\bf Y}(0) +a$, based on the observed data.  Because reallocation mimics random allocation, \eqref{additive} applies, hence
  \begin{equation}
 \var \left( \hat{\tau} \right) \rightarrow s_a^2 \left( \frac{1}{n_1} + \frac{1}{n_0} \right).
 \end{equation}
Under $ H_0: {\bf Y}(1) = {\bf Y}(0) +a$,
\begin{align*}
Y_i(1) &= Y_i^{obs}W_i + (Y_i^{obs} + a)(1-W_i), \text{and}
\\ 
Y_i(0) &= (Y_i^{obs}-a)W_i + Y_i^{obs}(1-W_i).
\end{align*}
WLOG, assume just for the following derivation that the units are ordered such that $W_i =  1$ for $i = 1, \dots, n_1$ and $W_i = 0$ for $i = n_1+1, \dots, n$.  Then
\begin{align}
s_a^2 &= \frac{ \sum_{i=1}^n \left(Y_i(1) - \overline{Y}(1) \right)^2}{n - 1} \notag \\
&= \frac{ \sum_{i=1}^n \left(Y_i(1) - \left( \overline{Y}^{obs} + \frac{n_0 a}{n} \right)  \right)^2}{n - 1} \notag \\
&= \frac{1}{n-1} \left[ \sum_{i=1}^{n_1} \left(Y_i^{obs} -  \overline{Y}^{obs} - \frac{n_0 a}{n}  \right)^2 +  \sum_{i=n_1+1}^{n} \left(Y_i^{obs} + a -  \overline{Y}^{obs} - \frac{n_0 a}{n}  \right)^2 \right] \notag \\
&= s^2  + \frac{a}{n (n-1)} \left[ \frac{(n_1 n_0^2 +n_0 n_1^2) a}{n} - 2n_0\sum_{i=1}^{n_1} (Y_i^{obs} -  \overline{Y}^{obs} ) + 2n_1 \sum_{i=n_1+1}^{n} (Y_i^{obs} -  \overline{Y}^{obs} ) \right]\notag \\
&= s^2  + \frac{n_0 n_1 a}{n (n-1)} \left(a - 2\hat{\tau} \right). \notag
\end{align}
\end{proof}

We make a few relevant observations regarding \eqref{ravar}.  Note that $(n_0 n_1)/(n(n-1))$ does not go zero as the sample size increases, hence $\var(\hat{\tau})$ always depends on the ratio of the group sizes, $\hat{\tau}$, and the hypothesized additive effect, $a$.   The variance of $\hat{\tau}$ is minimized when $a = \hat{\tau}$, and only equals the variance under the sharp null hypothesis of no treatment effect, \eqref{reallocatenullvar}, when $a = 0$ or $a = 2 \hat{\tau}$; $\var \left(\hat{\tau} \right)$ is lower than under the sharp null hypothesis for $0 < a < 2\hat{\tau}$, and higher otherwise.  Figure~\ref{additivevar} shows $\var \left(\hat{\tau} \right)$ as a function of $a$. 

\begin{figure}[!ht]
\begin{center}
\includegraphics[width=3.5in]{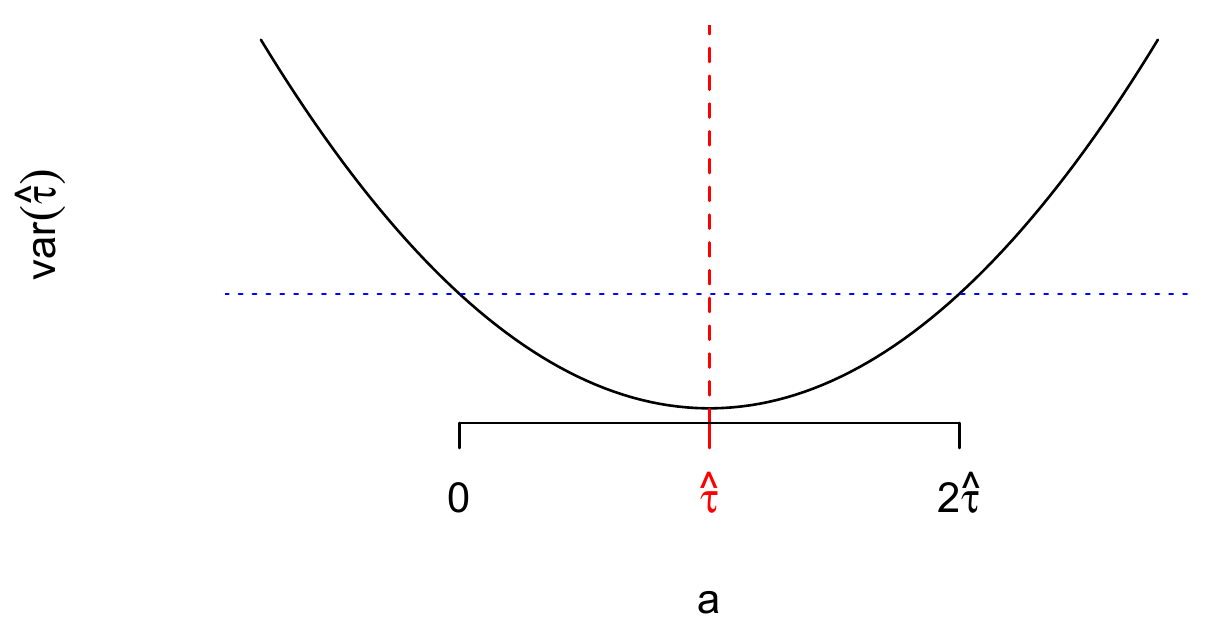}
\caption{The variance of $\hat{\tau}$ for reallocating under ${\bf Y}(1) = {\bf Y}(0)+a$, as a function of the assumed additive effect $a$.  The blue dotted line is $s^2 (1/n_1 + 1/n_0)$, or $\var(\hat{\tau})$ under the sharp null, when $a=0$.  The scales of the axes are not shown, as they depend on the data.  }
\label{additivevar}
\end{center}
\end{figure}

As an alternative to inverting a series of tests, some use the standard error from reallocating under the sharp null of $H_0: {\bf Y}(1) = {\bf Y}(0)$ to create the margin of error around the sample estimate.  However, this will generally not give the correct coverage, because $\var \left(\hat{\tau} \right)$ depends on $a$, and choosing to use the variance when $a = 0$ is arbitrary and generally incorrect, even for large samples.  If reallocation is to be used for estimation, the method of inverting many tests should be used, not the variance from one randomization distribution.

Without assuming additivity, estimation via reallocating becomes much more difficult \citep{kempthorne55}, and as far as we know, there is a not a good solution.  Reallocation should not be used for estimation when effects are suspected to vary substantially between units, or equivalently when within group variances differ substantially between groups.

\subsection{Random Sampling without a Null Hypothesis}

\begin{prop}
 Under random sampling of two groups of fixed sample sizes $n_0$ and $n_1$, 
 \begin{equation}\label{randsamplingvar}
 \var \left( \hat{\tau} \right) =  \frac{\sigma_1^2}{n_1} + \frac{\sigma_0^2}{n_0},
  \end{equation}
 where $\sigma_w^2$ is the population variance within group $w$.
 \end{prop}
 
 This is a commonly known result, and was also derived in \eqref{rsvar}.

\subsection{Resampling for Interval Estimation}\label{resampleint}

The resampling method we consider for estimation involves taking two independent samples with replacement: one of size $n_1$ taken from units with $W_i = 1$, and the other of size $n_0$ taken from units with $W_i = 0$.  This process is repeated many times, and $\hat{\tau}$ calculated for each.  There are many ways to use this to generate a confidence interval, such as using the estimated standard error and a $t$-value to create the margin of error or by generating an interval via percentiles \citep{diciccio88, diciccio96, hesterberg15}.  For comparison purposes, here we again focus only on $\var(\hat{\tau})$.

 \begin{prop}
 Under resampling of two groups of fixed sample sizes $n_0$ and $n_1$, as the number of simulations increases, 
\begin{equation}\label{resvar}
\var \left( \hat{\tau} \right) \rightarrow  \frac{s_1^2}{n_1}\left(\frac{n_1 - 1}{n_1}\right) + \frac{s_0^2}{n_0}  \left(\frac{n_0 - 1}{n_0}\right),
\end{equation}
 where $s_w^2 \equiv \frac{1}{n_w - 1} \sum_{W_i = w}  \left( Y_i^{obs} - \overline{Y}^{obs}(w) \right)^2$ is the sample variance within group $w$.   
\end{prop}

\begin{proof}
The result follows immediately by the analogous version of \eqref{bootmeanvar} and independence. \end{proof}

Again this is slightly biased below the true variance, \eqref{randsamplingvar}, but a consistent estimate.

\section{Testing a Null of Equal Means}\label{equalmeans}

Section~\ref{sharpnull} considers the sharp null hypothesis of no difference, but often we postulate a weaker null hypothesis, specifying only equality in the means: $H_0: \mu_0 = \mu_1$ for the population framework and $H_0: \overline{Y}(1) = \overline{Y}(0)$ for the finite sample framework.    How does this change the resulting standard errors? 

The true variances of $\hat{\tau}$ under this null, for both random allocation and random sampling, are equivalent to those derived for estimation, \eqref{ravar} and \eqref{randsamplingvar}, because both derivations are based on the variability of the data within groups (and across individual differences in potential outcomes for random allocation), and these do not change under equal means.  

As discussed in Section~\ref{reallocateint}, the process of reallocating requires a hypothesis that implies the values for all potential outcomes from the observed data.  Although this could be done without assuming additivity (by postulating a null difference for $Y_i(1) - Y_i(0)$ that is not constant across $i$), this is rarely, if ever, done in practice.  Therefore, to avoid the assumption of equal variances implicit with additivity, reallocating is essentially not a viable option.    

Philip Good has been reported as distinguishing testing via resampling and reallocating by the hypotheses themselves, noting that reallocating tests ``hypotheses concerning distributions," while resampling tests ``hypotheses concerning parameters" \citep[p. 7]{berry14}.  While reallocating is limited to testing the sharp null hypothesis, resampling offers the flexibility of either null hypothesis. Under the sharp null of no difference, units can be sampled from the combined sample, as discussed in Section~\ref{resampletest}.  Under the null of equal means, one can shift the groups to have equal means, creating new data, ${\bf Y}^*$, such that 
\begin{align}
Y_i^* \equiv \begin{cases}
Y_i^{obs} - \hat{\tau}/2 + b \text{ if } W_i = 1, \\
Y_i^{obs} + \hat{\tau}/2 + b \text{ if } W_i = 0,  
\end{cases}
\end{align} 
where $b$ can be any value.  Common choices for $b$ include $\hat{\tau}/2$, $-\hat{\tau}/2$, or $ 0$, but the value of $b$ is actually irrelevant for a difference in means because it will cancel out.  Resampling then occurs by resampling $n_1$ units from units with $W_i = 1$ and $n_0$ units from units with $W_i = 0$, and using the shifted values, $Y_i^*$ to calculate $\hat{\tau}$.  Because resampled units always come from their original group, all distributional differences between the groups besides center are preserved.  This is equivalent to resampling from the original values within each group as described in Section~\ref{resampleint} and then shifting the bootstrap distribution to be centered at 0, so the resulting variance is equivalent to that given in \eqref{resvar}.  


The specification of the null hypothesis makes much more of a difference than the choice to resample or reallocate, and the two decisions should not be conflated.  The resampling variances under a sharp null, \eqref{resnullvar}, and this weaker null, \eqref{resvar}, can differ substantially, while under a sharp null hypothesis it doesn't make much difference whether you reallocate, \eqref{reallocatenullvar}, or resample, \eqref{resnullvar}, at least in large samples.

\section{Two Examples}\label{examples}

We illustrate with two examples.  The first is based on data from a randomized experiment with relatively small sample sizes, approximately equal standard deviations between the groups, and relatively symmetric distributions.  The second is based on data from a random sample with larger sample sizes, standard deviations differing substantially between the groups, and positively skewed distributions with high outliers.  In each case we simulate each of the simulation procedures, using 20 million simulations to make the Monte Carlo variation negligible; in all cases the empirical standard errors are identical to the theoretical values as derived in Sections~\ref{sharpnull} to \ref{equalmeans}.  We also calculate empirical p-values or confidence intervals, which are all stable up to at least three decimal places.  

\subsection{Sleep versus Caffeine: A Randomized Experiment}\label{sc}

Here we analyze data from a randomized experiment comparing the effect of sleep versus caffeine on memory \citep{mednick08}.  People were shown a list of words to memorize, then randomly divided into two equally sized groups, $n_1 = n_0 = 12$.  One group ($W_i = 1$) took a nap for 90 minutes, while the other group ($W_i = 0$) took a caffeine pill and stayed awake.  The outcome, ${\bf Y}^{obs}$, is the number of words recalled after the nap/caffeine break.  The dataset, {\em SleepCaffeine}, is available at \url{www.lock5stat.com/datapage.html}.  The data are shown in Figure~\ref{scData}, and the sample difference in means is
\begin{equation}
\hat{\tau} = \overline{Y}^{obs}(1) - \overline{Y}^{obs}(0) = 15.25 - 12.25 = 3.
\end{equation} 

\begin{figure}[!ht]
\begin{center}
\includegraphics[width=.6\textwidth]{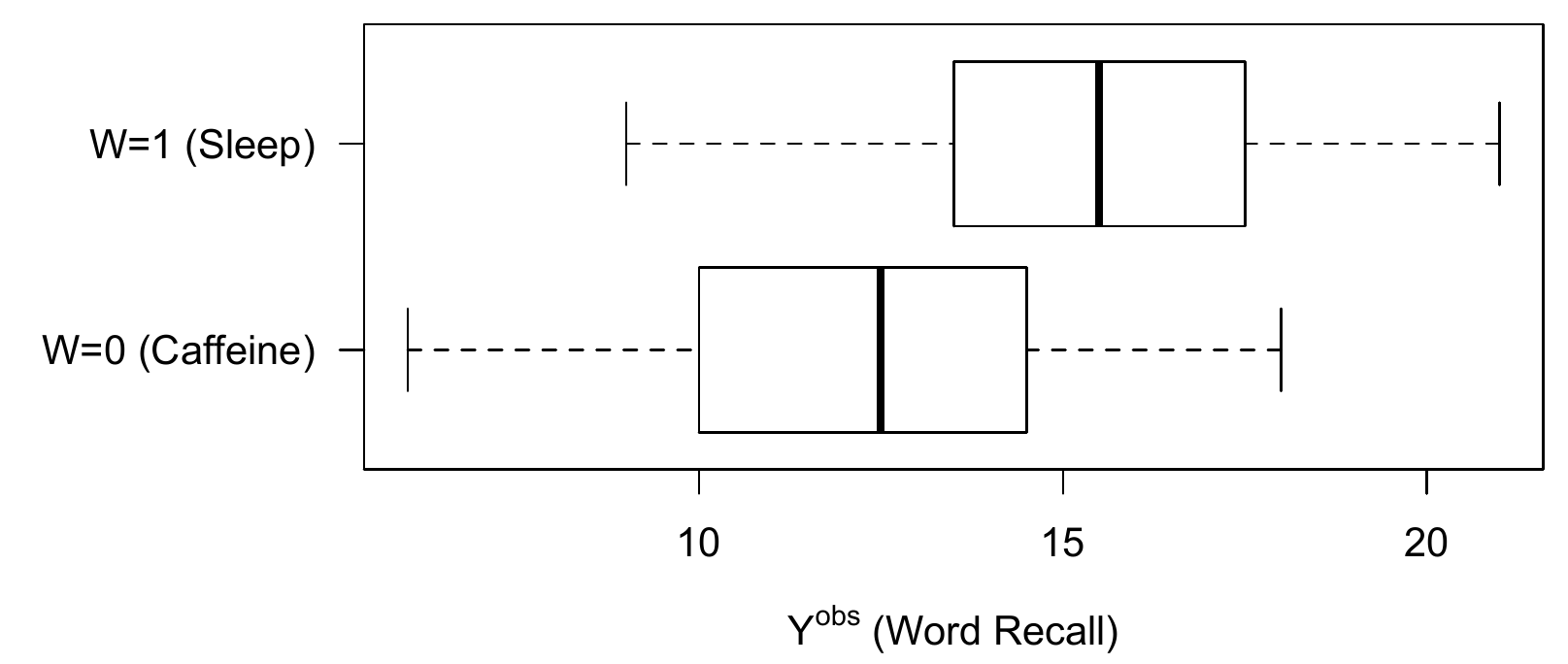}
\caption{Data from the sleep versus caffeine randomized experiment on memory.}
\label{scData}
\end{center}
\end{figure}

The overall sample standard deviation is $s = 3.686$, so the true standard error under random allocation and a sharp null hypothesis can be calculated by \eqref{ranullvar}:
\begin{equation*}
\left[ s^2 \left( \frac{1}{n_1} + \frac{1}{n_0} \right)\right]^{\frac{1}{2}} = \left[ 3.686^2 \left( \frac{1}{12} + \frac{1}{12} \right) \right]^{\frac{1}{2}} =1.505.
\end{equation*}
By \eqref{reallocatenullvar}, this is also the standard error for reallocating under the sharp null.  The standard error for resampling under the sharp null (from a combined sample), by \eqref{resnullvar}, is
\begin{equation*}
\left[ s^2 \left( \frac{1}{n_1} + \frac{1}{n_0} \right)\left( \frac{n-1}{n} \right) \right]^{\frac{1}{2}}  = \left[  3.686^2 \left( \frac{1}{12} + \frac{1}{12} \right)\left( \frac{24-1}{24} \right) \right]^{\frac{1}{2}} = 1.473.
\end{equation*}
By \eqref{resvar}, the standard error for resampling for intervals or a null of equal means is
\begin{equation*}
 \left[  \frac{s_1^2}{n_1}\left(\frac{n_1 - 1}{n_1}\right) + \frac{s_0^2}{n_0}  \left(\frac{n_0 - 1}{n_0}\right) \right]^{\frac{1}{2}} =\left[   \frac{3.306^2}{12}\left(\frac{12 - 1}{12}\right) + \frac{3.545^2}{12} \left(\frac{12 - 1}{12}\right) \right]^{\frac{1}{2}} = 1.340.
\end{equation*}
The standard errors for resampling are noticeablely smaller because of the small samples.  

One-sided p-values are calculated as the proportion of simulated samples yielding statistics greater than or equal to $\hat{\tau}$.  We also compute a 95\% confidence interval via reallocating (inverting tests), and resampling using $\hat{\tau} \pm t^*  SE$:
\begin{equation}
\hat{\tau} \pm t^* SE = 3 \pm 2.2(1.34) = (0.05, 5.95).
\end{equation}
Although not recommended, for comparison we also compute the interval using the standard error from the reallocating sharp null distribution.  Results are shown in Table~\ref{sleepResults}.

\begin{table}[!ht]
\begin{center}
\begin{tabular}{|c|l|c|c|c|c|}
\hline
&& SE & p-value or 95\% CI  \\
 \hline
\multirow{4}{*}{Test} & {\em Truth}: Random allocation (sharp null) & 1.505 & -  \\ \cline{2-4}
& Reallocating (sharp null) & 1.505 & 0.025 \\ \cline{2-4}
&Resampling (sharp null) & 1.473 & 0.022 \\ \cline{2-4}
& Resampling (equal means) & 1.340 & 0.013 \\ \hline
\multirow{3}{*}{Interval} & Reallocating (inverting test) & - & (0.00, 6.00) \\ \cline{2-4}
& Reallocate (sharp null SE) & 1.505 & (-0.31, 6.31) \\ \cline{2-4}
& Resampling & 1.340 & (0.05, 5.95) \\ \hline
\end{tabular}
\end{center}
\caption{Standard errors, p-values, and 95\% confidence intervals for each of the different simulation methods applied to the sleep versus caffeine example.}
\label{sleepResults}
\end{table}

\subsection{Income by Sex: A Random Sample} \label{acs}

Here we analyze data obtained via random sampling, looking at how income of employed American's differs by sex using a random subsample of employed adults from the American Community Survey \citep{acs10}.    Let $W_i = 1$ for males and $W_i = 0$ for females; $n_1 = 214$ and $n_0 = 217$. Let $Y^{obs}$ denote annual income in thousands of dollars.  The dataset, {\em ACS}, is available at \url{www.lock5stat.com/datapage.html}.  The data are displayed in Figure~\ref{acsData}, and the observed difference in means is
\begin{equation}
\hat{\tau} = \overline{Y}^{obs}(1) - \overline{Y}^{obs}(0) = 50.96 - 32.16 = 18.8.
\end{equation} 

\begin{figure}[!ht]
\begin{center}
\includegraphics[width=.6\textwidth]{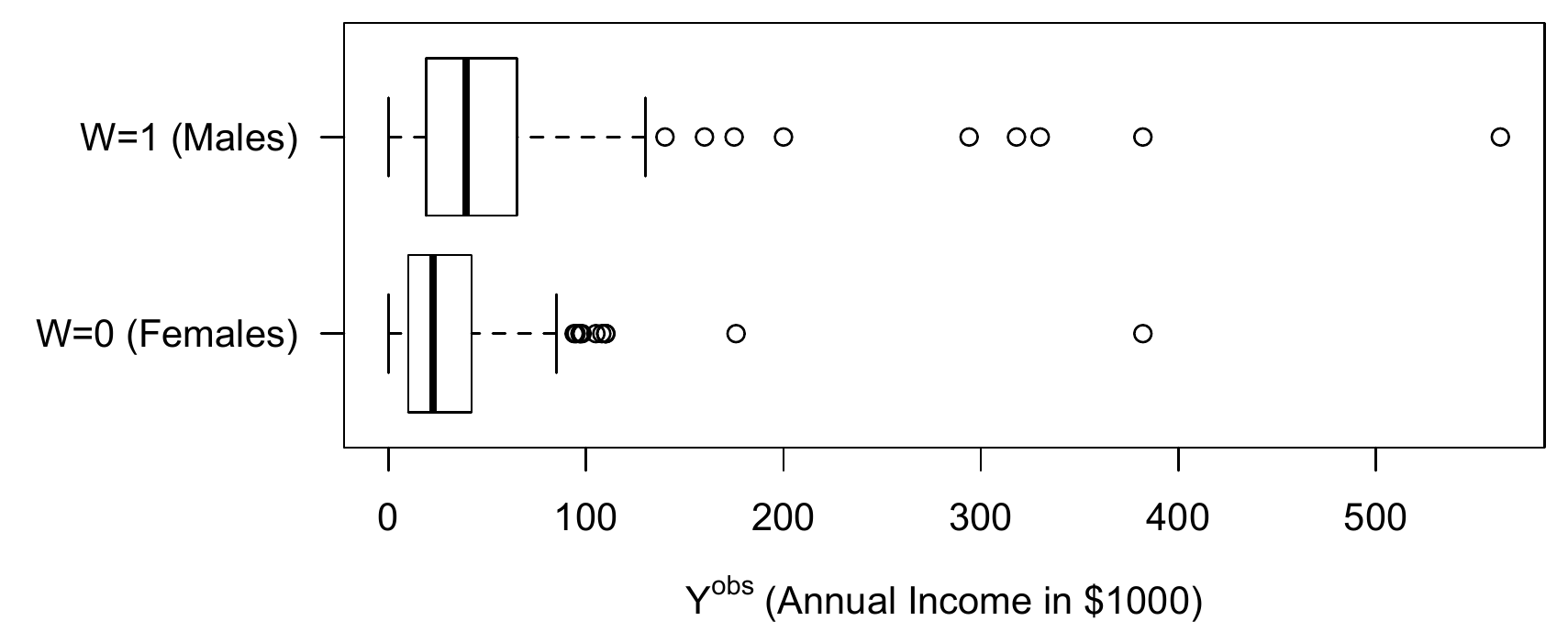}
\caption{Data from the American Community Survey on annual income by sex.}
\label{acsData}
\end{center}
\end{figure}

The theoretical standard errors are
\begin{equation*}
\left[ s^2 \left( \frac{1}{n_1} + \frac{1}{n_0} \right) \right]^{\frac{1}{2}} = \left[ 52.248^2 \left( \frac{1}{214} + \frac{1}{217} \right) \right]^{\frac{1}{2}} = 5.034
\end{equation*}
under reallocation with a sharp null,  
\begin{equation*}
\left[ s^2 \left( \frac{1}{n_1} + \frac{1}{n_0} \right)\left( \frac{n-1}{n} \right)\right]^{\frac{1}{2}} = \left[ 52.248^2 \left( \frac{1}{214} + \frac{1}{217} \right)\left( \frac{431-1}{431} \right)\right]^{\frac{1}{2}} = 5.028
\end{equation*}
under resampling from a combined population (sharp null), and
\begin{equation*}
\left[ \frac{s_1^2}{n_1}\left(\frac{n_1 - 1}{n_1}\right) + \frac{s_0^2}{n_0}  \left(\frac{n_0 - 1}{n_0}\right) \right]^{\frac{1}{2}}  =  \left[ \frac{62.848^2}{214}\left(\frac{214 - 1}{214}\right) + \frac{36.920^2}{217}  \left(\frac{217 - 1}{217}\right) \right]^{\frac{1}{2}} = 4.962
\end{equation*}
under resampling within (possibly shifted) groups.   Here the sample sizes are larger than in the previous example, so reallocating and resampling differ less under the sharp null.  However, there is more of a difference in the standard deviation between the two groups, thus resampling within groups differs from the sharp null standard errors.  

As with the previous example, we calculated p-values and confidence intervals, with results given in Table~\ref{acsResults}. These results are all highly significant, and we have strong evidence that among employed American adults, males make substantially more than females.    

\begin{table}[!ht]
\begin{center}
\begin{tabular}{|c|l|c|c|}
\hline
&& SE & p-value or 95\% CI  \\
 \hline
\multirow{3}{*}{Test} &  Reallocating (sharp null) & 5.034 & $2 \times 10^{-5}$ \\ \cline{2-4}
&Resampling (sharp null) & 5.028 & $1.5 \times 10^{-4}$ \\ \cline{2-4}
& Resampling (equal means) & 4.962 & $2.6 \times 10^{-4}$ \\ \hline
\multirow{3}{*}{Interval} & Reallocating (inverting test) & - & (9.13, 28.46) \\ \cline{2-4}
& Reallocate (sharp null SE) & 5.034 & (8.89, 28.72) \\ \cline{2-4}
& Resampling & 4.962 & (9.03, 28.58) \\ \hline
\end{tabular}
\end{center}
\caption{Standard errors, p-values, and 95\% confidence intervals for each of the different simulation methods applied to the income and sex example.}
\label{acsResults}
\end{table}

The testing results across methods are counterintuitive; the standard errors and p-values are inversely correlated!  This is because the simulated distributions are not entirely normal; as seen in Figure~\ref{acsData}, income is highly right-skewed.  In either simulation under the sharp null, the actual sex of participants (${\bf W}$) is irrelevant, and because $n_1 \approx n_0$, the two groups sampled are then indistinguishable, guaranteeing symmetry.  On the other hand, this does not hold when resampling within groups; because the males are more highly skewed with more high outliers, resampling within group yields a positively skewed distribution.  This means there will be more high values in the right tail, explaining how this distribution yields higher p-values, despite the lower standard error.  Here the Pearson Moment Coefficient of Skewness (estimated as $m_3/s^3$, where $m_3$ is the third central moment) is about 0.15. 

Although the standard errors under the sharp null appear almost identical, the p-values differ by an order of magnitude, again cautioning against the sole use of standard errors to judge comparability. To get at this difference, we have to go one moment further and examine kurtosis.  Based on our simulations, the sample excess kurtosis (estimated as $m_4/s^4 - 3$, where $m_4$ is the fourth central moment) is about -0.17 for reallocating and 0.08 for resampling (both sampling from the combined sample and within groups).  Because the distributions based on resampling have fatter tails, in part because resampling allows outliers to be sampled multiple times while reallocating does not, they can yield larger p-values despite smaller standard errors.  These issues were not relevant to the sleep versus caffeine example, despite the smaller sample sizes, because the original data is relatively symmetric and free of outliers (see Figure~\ref{scData}).  Further work comparing higher moments for reallocating and resampling would be worthwhile. 

\section{Summary}\label{summary}

Table~\ref{vartauhat} summarizes $\var(\hat{\tau})$ for each of the data collection methods (random allocation and random sampling) and simulation methods (reallocation and resampling), for testing a sharp null and generating confidence intervals.  The results under a null of equal means parallel those for intervals, except for reallocating.  

\begin{table}[!ht]
\begin{center}
\begin{tabular}{|l|l|l|}
\hline
Method & Testing: Sharp Null & Intervals  \\
\hline
Random Allocation&  $\displaystyle s^2 \left( \frac{1}{n_1} + \frac{1}{n_0}\right) $ & $\displaystyle  \frac{{s^*_1}^2}{n_1} + \frac{{s^*_0}^2}{n_0} - \frac{{s^*_{1-0}}^2}{n}$ \\
\hline
Reallocating & $\displaystyle s^2  \left( \frac{1}{n_1} + \frac{1}{n_0} \right) $ & $\displaystyle \left( s^2  + \frac{n_0 n_1 a}{n (n-1)} \left(a - 2\hat{\tau} \right) \right) \left( \frac{1}{n_1} + \frac{1}{n_0} \right)$ for each $a$ \\
\hline
Random Sampling &  $\displaystyle \sigma^2 \left( \frac{1}{n_1} + \frac{1}{n_0}\right)$ & $\displaystyle \frac{\sigma_1^2}{n_1} + \frac{\sigma_0^2}{n_0}$ \\ 
\hline
Resampling & $\displaystyle s^2 \left( \frac{1}{n_1} + \frac{1}{n_0} \right) \left( \frac{n-1}{n} \right)$ &  $\displaystyle  \frac{s_1^2}{n_1}\left(\frac{n_1 - 1}{n_1}\right) + \frac{s_0^2}{n_0}  \left(\frac{n_0 - 1}{n_0}\right)$ \\
\hline
\end{tabular}
\caption{Theoretical variance of $\hat{\tau}$ under different data collection and simulation methods.  Here $s^2$ and $s_w^2$ are known sample variances, $\sigma^2$ and $\sigma_w^2$ are unknown population variances, and ${s_w^*}^2$ are unknown sample variances calculated across all units.}
\label{vartauhat}
\end{center}
\end{table}


When testing a sharp null, all methods yield very similar variances.  The variance of $\hat{\tau}$ from random sampling differs from that of random allocation only in that it uses $\sigma^2$ rather than $s^2$.  With enough simulations, reallocation converges to the exact null distribution that would be observed under random allocation, providing a gold standard test.  The variance of $\hat{\tau}$ under resampling only differs from reallocating by a factor of $(n-1)/n$, which is negligible for large $n$, but induces bias for small $n$.   The variance from reallocating gives an unbiased estimate of the true variance of $\hat{\tau}$, regardless of whether the data were collected using random allocation or random sampling, while the variance from resampling gives a consistent estimate of the true variance, but biased slightly too low.   In short, when testing a sharp null hypothesis reallocation is slightly superior to resampling, regardless of the method of data collection, although they are asymptotically equivalent.  

When generating a confidence interval, the differences in $\var(\hat{\tau})$ across both data collection methods and simulation methods are more substantial.  The true variance of $\hat{\tau}$ under random allocation involves the variance of the unit-level treatment effects, and thus is impossible to estimate unbiasedly using only the observed data.  It will in general be lower than $\var(\hat{\tau})$ under random sampling.  Reallocating requires an assumed relationship between ${\bf Y}(1)$ and ${\bf Y}(0)$ in order to fill in all the missing potential outcomes from ${\bf Y}^{obs}$, which usually takes the form of assuming additivity, ${\bf Y}(1) = {\bf Y}(0) + a$. Inverting many randomization tests for different values of $a$ can yield a valid confidence interval, but only under the assumption of additivity, or equivalently equal variances between treatment groups.  Using the standard error from the sharp null randomization distribution for estimation is not recommended, as the standard error depends on the value of the hypothesized additive effect, $a$, and this approach arbitrarily sets $a=0$.  The variance of $\hat{\tau}$ from resampling is a consistent estimate for the true variance under random sampling, and a conservative estimate of the true variance under random allocation, which is the best that can be done without additional assumptions.  For intervals, resampling is superior to reallocation, regardless of the method of data collection, unless the treatment effect is truly additive.

These results are counterintuitive, and go contrary to popular belief.  Intuition would lead one to believe that the randomization in the simulation method should mimic the randomization in the data collective method (if it matters at all), but this is not the case.  For a difference in means, with either complete randomization to two groups or simple random sampling of two groups, we find that reallocating is superior for testing a sharp null hypothesis with small $n$ (while for large $n$ it doesn't matter which method is used), that resampling is superior for intervals unless additivity holds, and that these findings apply {\em regardless of how the data were collected}.

Of course, for more complex study designs, it becomes more important that the simulation method mimic the data collection method, and for more complex designs reallocating will likely be more natural if random allocation was used and resampling more natural if random sampling was used.  Likewise, it is always important that randomness be present in some form, and hence the data collection method will always be relevant for inference.  Moreover, our results are based solely on $\var(\hat{\tau})$, but, as evidenced by Section~\ref{acs}, further moments should be considered if $\hat{\tau}$ is not normally distributed.  However, for the cases examined here, matching the randomness in simulation to the randomness in data collection need not be a priority, at least not for mathematical reasons.

\bibliographystyle{agsm.bst}

\bibliography{simulationBib}
\end{document}